\documentclass[11pt]{article}

\usepackage{amssymb}
\usepackage{amsmath}
\usepackage{amsthm}
\usepackage{tikz-cd}
\usepackage{soul}
\usepackage{xcolor}

\numberwithin{equation}{section}

\newtheoremstyle{slplain}
  {\topsep}
  {\topsep}
  {\slshape}
  {0pt}
  {\bfseries}
  {.}
  {0.5em}
  {}

\theoremstyle{slplain}
  \newtheorem{THM}{Theorem}[section]
  \newtheorem{LEM}[THM]{Lemma}
  \newtheorem{PROP}[THM]{Proposition}
  \newtheorem{COR}[THM]{Corollary}

\theoremstyle{definition}

\newcommand\toCC[1]{\overset{#1}\longrightarrow}

\renewcommand{\le}{\leqslant}
\renewcommand{\ge}{\geqslant}

\newcommand{\0}{\varnothing}

\renewcommand{\sec}{\cap}
\renewcommand{\phi}{\varphi}
\renewcommand{\epsilon}{\varepsilon}

\newcommand{\BB}{\mathbf{B}}
\newcommand{\CC}{\mathbf{C}}
\newcommand{\DD}{\mathbf{D}}
\newcommand{\GG}{\mathbf{G}}

\newcommand{\KK}{\mathbf{K}}

\newcommand{\NN}{\mathbb{N}}
\newcommand{\PP}{\mathbf{P}}

\renewcommand{\SS}{\mathbf{S}}

\newcommand{\union}{\cup}
\newcommand{\restr}[2]{\hbox{$#1$}\hbox{$\upharpoonright$}_{#2}}
\newcommand{\reduct}[2]{\hbox{$#1$}\hbox{$|$}_{#2}}
\newcommand{\Boxed}[1]{\mbox{$#1$}}
\newcommand{\id}{\mathrm{id}}

\newcommand{\Ob}{\mathrm{Ob}}

\newcommand{\op}{\mathrm{op}}

\newcommand{\calA}{\mathcal{A}}
\newcommand{\calB}{\mathcal{B}}
\newcommand{\calC}{\mathcal{C}}

\newcommand{\calF}{\mathcal{F}}

\newcommand{\calX}{\mathcal{X}}

\newcommand{\Set}{\mathbf{Set}}

\newcommand{\iso}{\mathrm{iso}}

\newcommand{\Aut}{\mathrm{Aut}}

\newcommand{\subobj}[3]{\binom{#3}{#2}_{#1}}

\newcommand\toDD{\overset{\DD}\longrightarrow}
\newcommand\toCCIII{\overset{\CC_1 \times \CC_2}\longrightarrow}

\newcommand{\ChEmb}{\mathbf{FinChn}}
\newcommand{\ChRSurj}{\mathbf{FinChn}_{\mathit{rs}}}
\newcommand{\slice}[2]{(#1 \downarrow #2)}

\title{Ramsey properties of products and pullbacks of categories and the Grothendieck construction\\
      (with corrections)}
\author{%
  Dragan Ma\v sulovi\'c\\
  University of Novi Sad, Faculty of Sciences\\
  Department of Mathematics and Informatics\\
  Trg Dositeja Obradovi\'ca 3, 21000 Novi Sad, Serbia\\
  e-mail: masul@dmi.uns.ac.rs}

\begin{document}
\maketitle

\begin{abstract}
  In this paper we provide purely categorical proofs of two important results of structural Ramsey theory:
  the result of M.\ Soki\'c that the free product of Ramsey classes is a Ramsey class,
  and the result of \textcolor{red}{M.\ Bodirsky, M.\ Pinsker and T.\ Tsankov} that adding constants to the language of a Ramsey class preserves the Ramsey
  property. The proofs that we present here ignore the model-theoretic
  background of these statements. Instead, they focus on categorical constructions by which the classes can be
  constructed generalizing the original statements along the way.
  It turns out that the restriction to classes of relational structures,
  although fundamental for the original proof strategies, is not relevant for the statements themselves.
  The categorical proofs we present here remove all restrictions
  on the signature of first-order structures and provide the information not
  only about the Ramsey property but also about the Ramsey degrees.
    
  \bigskip


  \bigskip

  \noindent \textbf{AMS Subj.\ Classification (2010):} 18A99; 05C55
\end{abstract}

\section{Introduction}

Ramsey theory is one of only a few mathematical theories whose reach goes from elementary problems
that can be demonstrated to undergraduates, to a variety of surprisingly subtle applications in
mathematical logic, set theory, finite combinatorics and topological dynamics.
The successful generalization of Ramsey's theorem from sets (unstructured objects) to cardinals (special well-ordered chains)
lead in the early 1970's to generalizing this setup to arbitrary first-order structures, giving birth to
\emph{structural Ramsey theory}. Interestingly, one of the first comprehensive texts published already in 1973
was Leeb's book \cite{leeb-vorlesungen} where structural Ramsey theory was presented using the language of category theory.
Unfortunately, this point of view was relatively quickly pushed out of the focus of the research community
and was replaced by the model-theoretic approach that is even today a dominant point of view.

As the structural Ramsey theory evolved, it has become evident that the Ramsey phenomena
depend not only on the choice of objects but also on the choice of morphisms
prompting, thus, the shift of the attention back to categorical interpretations of the Ramsey phenomena.
However, instead of pursuing the original approach by Leeb (which has very fruitfully been applied to a wide range of
Ramsey-type problems \cite{GLR, leeb-cat, Nesetril-Rodl}),
we proposed in~\cite{masulovic-ramsey} a systematic study of
a simpler approach motivated by and implicit in \cite{mu-pon,vanthe-more,Zucker-1}.
This paper provides another argument that in some cases categorical reinterpretation of the Ramsey-related
phenomena can be beneficial.

In this paper we provide purely categorical proofs of two important results of structural Ramsey theory:
the result of M.\ Soki\'c that the free product of Ramsey classes is a Ramsey class~\cite{sokic2},
and the result of \textcolor{red}{M.\ Bodirsky, M.\ Pinsker and T.\ Tsankov} that adding constants to the language of a Ramsey class preserves the Ramsey
property~\textcolor{red}{\cite{bodirsky-pinsker-tsankov}} \st{\hbox{\cite{boridsky-ramsey.classes}}}.
The proofs that we present here ignore the model-theoretic
background of these statements. Instead, they focus on categorical constructions by which the classes can be
constructed generalizing the original statements along the way.

In Section~\ref{rpppg.sec.prelim} we recall some basic notions and fix some notation.

In Section~\ref{rpppg.sec.subca-ftr} we upgrade certain results from~\cite{masul-drp} and \cite{masul-kpt}
to obtain more versatile tools for transporting the Ramsey property via functors and from 
a category onto its subcategory.

In Section~\ref{rpppg.sec.prod-pb} we consider the behavior of the Ramsey property, and in particular
Ramsey degrees under products and pullbacks of categories. This enables us to prove a
significant generalization of the following result.

The Finite Product Ramsey Theorem~(see \cite[Theorem 5, p.\ 113]{GRS})
was generalized to products of Ramsey classes of finite structures by M.~Soki\'c~\cite{sokic2},
and later M.~Bodirsky provided a model-theoretic interpretation of the same result in~\cite[Proposition~3.3]{boridsky-ramsey.classes}.
Let $\KK_0, \ldots, \KK_{s-1}$ be classes
of finite structures and let $(\calA_i)_{i < s}, (\calB_i)_{i < s} \in \prod_{i < s} \KK_i$.
Then $\binom{(\calB_i)_{i < s}}{(\calA_i)_{i < s}}$ denotes the set of all
$(\calA'_i)_{i < s} \in \prod_{i < s} \KK_i$ such that $\calA'_i \cong \calA_i$ and
$\calA'_i$ is a substructure of $\calB_i$ for all $i < s$.

\begin{THM}[Free product of Ramsey classes]\cite[Theorem 2]{sokic2}
  Let $s, r \in \NN$, and let $\KK_0, \ldots, \KK_{s-1}$ be classes of finite structures
  with the structural Ramsey property. Fix two sequences $(\calA_i)_{i < s}, (\calB_i)_{i < s}
  \in \prod_{i < s} \KK_i$. Then there is a sequence $(\calC_i)_{i < s} \in \prod_{i < s} \KK_i$
  such that for any coloring $p : \binom{(\calC_i)_{i < s}}{(\calA_i)_{i < s}} \to r$
  there exist a number $l < r$ and a sequence $(\calB'_i)_{i < s} \in \prod_{i < s} \KK_i$ such that
  $\calB'_i \cong \calB_i$ for all $i < s$ and
  $p(\calX) = l$ for all $\calX \in \binom{(\calB'_i)_{i < s}}{(\calA_i)_{i < s}}$.
\end{THM}

In the discussion that follows we also relate to a result of M.\ Bodirsky
about strong amalgamation classes of finite structures with the Ramsey property~\cite{Bodirsky}.

Finally, in Section~\ref{rpppg.sec.grothendieck} we consider the behavior of the Ramsey property
under the Grothendieck construction and in slice categories and prove a generalization of the
following result. Let $\Theta$ be a relational signature.
Following~\cite{bodirsky-pinsker-tsankov} we shall say that a countable $\Theta$-structure $\calF$ is
\emph{Ramsey} if the class of all finitely generated substructures of $\calF$ has the Ramsey property.

\begin{THM}[Adding constants to the language]\cite{bodirsky-pinsker-tsankov}
  Adding constants to a relational language preserves the Ramsey property. More precisely, if $\calF$ is
  ordered, homogeneous and Ramsey, and if $x_1, \ldots, x_n$ are elements of $\calF$ then $\calF^* = (\calF, x_1, \ldots, x_n)$ is
  also ordered, homogeneous and Ramsey,
  where $\calF^*$ is a structure over a signature obtained from $\Theta$ by adding $n$ new constant symbols.
\end{THM}

It turns out that the restriction to classes of relational structures,
although fundamental for both model-theoretic and combinatorial proof strategies,
are not relevant for the statements themselves.
The categorical proofs we present here remove all restrictions
on the signature of first-order structures and provide the information not
only about the Ramsey property but also about the Ramsey degrees.

\section{Preliminaries}
\label{rpppg.sec.prelim}

\paragraph{Categories.}
Let us quickly fix some notation. Let $\CC$ be a category. By $\Ob(\CC)$ we denote the class of all the objects in $\CC$.
Homsets in $\CC$ will be denoted by $\hom_\CC(A, B)$, or simply $\hom(A, B)$ when $\CC$ is clear from the context.
The identity morphism will be denoted by $\id_A$ and the composition of morphisms by $\Boxed\cdot$ (dot).
If $\hom_\CC(A, B) \ne \0$ we write $A \toCC\CC B$ or simply $A \to B$.
Let $\iso_\CC(A, B)$ denote the set of all the invertible morphisms $A \to B$, and let
$\Aut_\CC(A) = \iso(A, A)$ denote the set of all the invertible morphisms $A \to A$.
An object $A \in \Ob(\CC)$ is \emph{rigid} if $\Aut_\CC(A) = \{\id_A\}$.

A category $\CC$ is \emph{directed} if for all $A, B \in \Ob(\CC)$ there exists a $C \in \Ob(\CC)$
such that $A \to C$ and $B \to C$; and it has \emph{amalgamation} if for all
$A, B_1, B_2 \in \Ob(\CC)$, $f_1 \in \hom(A, B_1)$ and $f_2 \in \hom(A, B_2)$
there is a $C \in \Ob(\CC)$ together with $g_1 \in \hom(B_1, C)$ and $g_2 \in \hom(B_2, C)$
such that $g_1 \cdot f_1 = g_2 \cdot f_2$.
As usual, $\CC^\op$ is the opposite category.

If $F : \CC \to \DD$ is a functor then $F(\CC)$ is a subcategory of $\DD$ whose
objects are of the form $F(C)$ where $C \in \Ob(\CC)$ and morphisms are
of the form $F(f)$ where $f$ is a morphism in~$\CC$.

\paragraph{Structures.}
A \emph{signature} is a set $\Theta = \Theta_F \union \Theta_R \union \Theta_C$ where $\Theta_F$
is a set of \emph{function symbols}, $\Theta_R$ is a set of \emph{relation symbols}
and $\Theta_C$ is a set of \emph{constant symbols} $(\Theta_C)$. A \emph{first-order structure of signature $\Theta$} or
a \emph{$\Theta$-structure} $\calA = (A, \Theta^\calA)$ is a set $A$ together with a set $\Theta^\calA$ of
functions on $A$, relations on $A$ and constants from $A$ which are interpretations of the corresponding symbols in $\Theta$.
The underlying set of a structure $\calA$, $\calA_1$, $\calA^*$, \ldots\ will always be denoted by its roman
letter $A$, $A_1$, $A^*$, \ldots\ respectively.
A structure $\calA = (A, \Theta^\calA)$ is \emph{finite} if $A$ is a finite set.
A signature $\Theta = \Theta_F \union \Theta_R \union \Theta_C$ is \emph{relational} if $\Theta_F = \Theta_C = \0$.
A \emph{relational structure} is a $\Theta$-structure where $\Theta$ is a relational signature.

If $\calA$ is a $\Theta$-structure and $\Sigma \subseteq \Theta$ then by $\reduct \calA \Sigma$ we denote
the \emph{$\Sigma$-reduct} of~$\calA$: $\reduct \calA \Sigma = (A, \{\theta^\calA : \theta \in \Sigma\})$.

Let $\Boxed< \notin \Theta$ be a binary relational symbol.
A \emph{finite linearly ordered $\Theta$-structure} is a
$(\Theta\union\{\Boxed<\})$-structure of the form $\calA = (A, \Theta^\calA, \Boxed{<^\calA})$ where
$<^\calA$ is a linear order on~$A$.

An \emph{embedding} $f: \calA \rightarrow \calB$ is an injection $f: A \rightarrow B$ which preserves
functions and constants, and respects and reflects relations. Surjective embeddings are \emph{isomorphisms}.
We write $\calA \cong \calB$ to denote that $\calA$ and $\calB$
are isomorphic, and $\calA \hookrightarrow \calB$ to denote that there is an embedding of $\calA$ into $\calB$.
A structure $\calA$ is a \emph{substructure} of a structure
$\calB$ ($\calA \le \calB$) if the identity map is an embedding of $\calA$ into $\calB$.
Let $\calA$ be a structure and $\0 \ne B \subseteq A$. Then $\calA[B] = (B, \restr {\Theta^\calA} B)$ denotes
the \emph{substructure of $\calA$ induced by~$B$}, where $\restr {\Theta^\calA} B$ denotes the restriction of
$\Theta^\calA$ to~$B$.
Note that $\calA[B]$ is not required to exist for every $B \subseteq A$. For example, if $\Theta$ contains function
symbols, only those $B$ which are closed with respect to all the functions in $\Theta^\calA$ qualify for the base set of a substructure.
However, if $\calA$ is a relational structure then $\calA[B]$ exists for every $B \subseteq A$.

Every class of finite structures $\KK$ can be thought of as a category where morphisms are embeddings and the
composition of morphisms is just the usual function composition.

\paragraph{Ramsey phenomena in a category.}
Let $\CC$ be a locally small category and let $A, B \in \Ob(\CC)$.
Write $f \sim_A g$ to denote that there is an $\alpha \in \Aut(A)$ such that $f = g \cdot \alpha$.
It is easy to see that $\sim_A$ is an equivalence relation on $\hom(A, B)$, so we let
$
  \binom BA = \hom(A, B) / \Boxed{\sim_A}
$
be the set of all the \emph{subobjects of $B$ isomorphic to $A$}.

For a $k \in \NN$, a \emph{$k$-coloring} of a set $S$ is any mapping $\chi : S \to K$, where
$K$ is a finite set with $|K| = k$. Often it will be convenient to take
$k = \{0, 1,\ldots, k-1\}$ as the set of colors. Clearly, a $k$-coloring of $S$
can also be understood as a union $S = X_1 \cup \ldots \cup X_k$ where
$i \ne j \Rightarrow X_i \cap X_j = \0$.

For an integer $k \in \NN$ and $A, B, C \in \Ob(\CC)$ we write
$
  C \overset\sim\longrightarrow (B)^{A}_{k, t}
$
to denote that for every $k$-coloring
$
  \chi : \subobj{} AC \to k
$
there is a morphism $w : B \to C$ such that $|\chi(w \cdot \subobj{} AB)| \le t$.
(Note that $w \cdot (f / \Boxed{\sim_A}) = (w \cdot f) / \Boxed{\sim_A}$ for $f / \Boxed{\sim_A} \in \subobj{}AB$;
therefore, we shall simply write $w \cdot f / \Boxed{\sim_A}$.)
Instead of $C \overset\sim\longrightarrow (B)^{A}_{k, 1}$ we simply write
$C \overset\sim\longrightarrow (B)^{A}_{k}$.

For $A \in \Ob(\CC)$ let $\tilde t_\CC(A)$ denote the least positive integer $n$ such that
for all $k \in \NN$ and all $B \in \Ob(\CC)$ there exists a $C \in \Ob(\CC)$ such that
$C \overset\sim\longrightarrow (B)^{A}_{k, n}$, if such an integer exists.
Otherwise put $\tilde t_\CC(A) = \infty$. Then $\tilde t_\CC(A)$ is referred to as the
\emph{structural Ramsey degree of $A$ in $\CC$}. 
A category $\CC$ \emph{has finite structural Ramsey degrees} if $\tilde t_\CC(A) < \infty$ for all $A \in \Ob(\CC)$.
An $A \in \Ob(\CC)$ is a \emph{Ramsey object in $\CC$} if $\tilde t_\CC(A) = 1$.
A locally small category $\CC$ has the
\emph{structural Ramsey property} if $\tilde t_\CC(A) = 1$ for all $A \in \Ob(\CC)$.
A locally small category $\CC$ has the
\emph{dual structural Ramsey property} if $\tilde t_{\CC^\op}(A) = 1$ for all $A \in \Ob(\CC)$.

The notion of embedding Ramsey degrees can be introduced analogously and proves to be much more
convenient when it comes to actual calculations. We write
$
  C \longrightarrow (B)^{A}_{k,t}
$
to denote that for every $k$-coloring
$
  \chi : \hom(A, C) \to k
$
there is a morphism $w : B \to C$ such that $|\chi(w \cdot \hom(A, B))| \le t$.
Instead of $C \longrightarrow (B)^{A}_{k, 1}$ we simply write
$C \longrightarrow (B)^{A}_{k}$.

For $A \in \Ob(\CC)$ let $t_\CC(A)$ denote the least positive integer $n$ such that
for all $k \in \NN$ and all $B \in \Ob(\CC)$ there exists a $C \in \Ob(\CC)$ such that
$C \longrightarrow (B)^{A}_{k, n}$, if such an integer exists.
Otherwise put $t_\CC(A) = \infty$. Then $t_\CC(A)$ is referred to as the
\emph{embedding Ramsey degree of $A$ in $\CC$}. 
A locally small category $\CC$ has the
\emph{embedding Ramsey property} if $t_\CC(A) = 1$ for all $A \in \Ob(\CC)$.
A locally small category $\CC$ has the
\emph{dual embedding Ramsey property} if $t_{\CC^\op}(A) = 1$ for all $A \in \Ob(\CC)$.

If all the objects in a category $\CC$ are rigid then the structural and embedding (dual) Ramsey
properties coincide and we speak simply of the Ramsey property.
Let us state two fundamental results of Ramsey theory.

\begin{THM}[Finite Ramsey Theorem]\cite{Ramsey}
  The category $\ChEmb$ whose objects are finite chains ($=$ linearly ordered sets) and morphisms
  are embeddings has the Ramsey property.
\end{THM}

For the formulation of the Finite Dual Ramsey theorem we have to introduce special surjective maps
between finite chains. Let $(A, \Boxed{<_A})$ and $(B, \Boxed{<_B})$ be tow finite chains and $f : A \to B$
a surjective mapping. We say that $f$ is a \emph{rigid surjection} if $\min f^{-1}(b) <_A \min f^{-1}(b')$
whenever $b <_B b'$.

\begin{THM}[Finite Dual Ramsey Theorem]\cite{GR}
  The category $\ChRSurj$ whose objects are finite chains ($=$ linearly ordered sets) and morphisms
  are rigid surjections has the dual Ramsey property.
\end{THM}

The embedding Ramsey property in a category can force other structural properties.
In particular,

\begin{THM}\cite{Nesetril}\label{crt.thm.RP=>AP}
  Let $\CC$ be a locally small directed category with the embedding Ramsey property.
  Then $\CC$ has amalgamation.
\end{THM}

The following relationship between structural and embedding Ramsey degrees
was proved for relational structures in \cite{Zucker-1} and generalized to this form in \cite{masul-kpt}.

\begin{PROP}\label{rdbas.prop.sml} (\cite{Zucker-1,masul-kpt})
  Let $\CC$ be a locally small category whose morphisms are mono
  and let $A \in \Ob(\CC)$. Then $t(A)$ is finite if and only if both $\tilde t(A)$ and $\Aut(A)$ are finite,
  and in that case
  $
    t(A) = |\Aut(A)| \cdot \tilde t(A)
  $.
\end{PROP}

The following lemma will be useful in the sequel.

\begin{LEM}{[folklore]}\label{cerp.lem.easy}
  Let $\CC$ be a locally small category whose morphisms are mono. Let $k, t \in \NN$ be positive integers and
  $A, B, C, D \in \Ob(\CC)$.

  $(a)$
  If $C \overset\sim\longrightarrow (B)^{A}_{k,t}$ and $D \to B$ then $C \overset\sim\longrightarrow (D)^{A}_{k, t}$.

  $(b)$
  If $C \longrightarrow (B)^{A}_{k, t}$ and $D \to B$ then $C \longrightarrow (D)^{A}_{k, t}$.

  $(c)$
  If $C \overset\sim\longrightarrow (B)^{A}_{k, t}$ and $C \to D$ then $D \overset\sim\longrightarrow (B)^{A}_{k, t}$.

  $(d)$
  If $C \longrightarrow (B)^{A}_{k, t}$ and $C \to D$ then $D \longrightarrow (B)^{A}_{k, t}$.
\end{LEM}

\paragraph{Convention ($\dagger$).} 
Let $\NN_\infty = \NN \union \{\infty\} = \{1, 2, 3, \ldots, \infty \}$.
The usual linear order on the positive integers extends to $\NN_\infty$ straightforwardly:
$
  1 < 2 < \ldots < \infty
$.
Ramsey degrees take their values in $\NN_\infty$, so when we write
$t_1 \ge t_2$ for some Ramsey degrees $t_1$ and $t_2$ then
  $t_1, t_2 \in \NN$ and $t_1 \ge t_2$; or
  $t_1 = \infty$ and $t_2 \in \NN$; or
  $t_1 = t_2 = \infty$.
For notational convenience, if $A$ is an infinite set we shall simply write
$|A| = \infty$ regardless of the actual cardinal~$|A|$. Hence, if $t$ is a Ramsey degree
and $A$ is a set, by $t \ge |A|$ we mean the following:
  $t \in \NN$, $|A| \in \NN$ and $t \ge |A|$; or
  $t = \infty$ and $|A| \in \NN$; or
  $A$ is an infinite set and $t = \infty$.
On the other hand, if $A$ and $B$ are sets then $|A| \ge |B|$ has the usual meaning.

With this convention in mind Proposition~\ref{rdbas.prop.sml} takes the following much simpler form:
$t(A) = |\Aut(A)| \cdot \tilde t(A)$ for all $A \in \Ob(\CC)$.

\section{Subcategories and functors}
\label{rpppg.sec.subca-ftr}

In this section we consider the behavior of Ramsey degrees under functors and in subcategories.
Simpler versions of the results presented here can be found in \cite{masul-drp} and \cite{masul-kpt}.
However, for the results presented in Sections~\ref{rpppg.sec.prod-pb} and~\ref{rpppg.sec.grothendieck}
we need the generalizations that we now prove.

A functor $F : \CC \to \DD$ \emph{preserves automorphism groups}
if $F(\Aut_\CC(A)) = \Aut_\DD(F(A))$ for all $A \in \Ob(\CC)$.

\begin{LEM}\label{rament.lem.cofin-full-faith-weak}
  Let $\CC$ and $\DD$ be locally small categories whose morphisms are mono, and let
  $F : \CC \to \DD$ be a full functor.
  
  $(a)$ If $C \longrightarrow (B)^A_{k, t}$ for some $A, B, C \in \Ob(\CC)$ and $k, t \in \NN$
  then $F(C) \longrightarrow (F(B))^{F(A)}_{k, t}$.
  
  $(b)$ If $F$ preserves automorphism groups and $C \overset\sim\longrightarrow (B)^A_{k, t}$ for some
  $A, B, C \in \Ob(\CC)$ and $k, t \in \NN$ then $F(C) \overset\sim\longrightarrow (F(B))^{F(A)}_{k, t}$.

  $(c)$ If $F(\CC)$ is cofinal in $\DD$ then $t_\DD(F(A)) \le t_\CC(A)$ for all $A \in \Ob(\CC)$.

  $(d)$ If $F$ preserves automorphism groups
  and $F(\CC)$ is cofinal in $\DD$ then $\tilde t_\DD(F(A)) \le \tilde t_\CC(A)$
  for all $A \in \Ob(\CC)$.
\end{LEM}
\begin{proof}
  $(a)$ Take any coloring $\chi : \hom_\DD(F(A), F(C)) \to k$ and define $\chi' : \hom_\CC(A, C) \to k$
  by $\chi'(f) = \chi(F(f))$. Since $C \longrightarrow (B)^A_{k, t}$ there is a $w \in \hom_\CC(B, C)$
  such that $|\chi'(w \cdot \hom_\CC(A, B))| \le t$. Now, $F$ is a full functor, so
  $$
    F(w) \cdot \hom_\DD(F(A), F(B)) = F(w \cdot \hom_\CC(A, B)),
  $$
  whence
  $
    |\chi(F(w) \cdot \hom_\DD(F(A), F(B)))| = |\chi(F(w \cdot \hom_\CC(A, B)))|
    = |\chi'(w \cdot \hom_\CC(A, B))| \le t
  $.
  
  \medskip

  $(b)$ Let $F$ be a full functor which preserves automorphism groups.

  \medskip

  Claim 1. $F(f / \Boxed{\sim_A}) = F(f) / \Boxed{\sim_{F(A)}}$ for all $A, B \in \Ob(\CC)$ and $f \in \hom_\CC(A, B)$.

  Proof. Let us only prove inclusion $(\supseteq)$. Take any $g \in F(f) / \Boxed{\sim_{F(A)}}$.
  Then $g = F(f) \cdot \beta$ for some $\beta = \Aut_\DD(F(A))$. Since 
  $\Aut_\DD(F(A)) = F(\Aut_\CC(A))$ there is an $\alpha \in \Aut_\CC(A)$ such that
  $F(\alpha) = \beta$, so $g = F(f) \cdot F(\alpha) = F(f \cdot \alpha) \in F(f / \Boxed{\sim_A})$.

  \medskip

  Claim 2. $F(w \cdot \subobj{} AB) = F(w) \cdot \subobj{}{F(A)}{F(B)}$ for all $A, B, C \in \Ob(\CC)$ and
  $w \in \hom_\CC(B, C)$.

  Proof. Let us only prove inclusion $(\supseteq)$. Take any
  $g \in F(w) \cdot \subobj{}{F(A)}{F(B)}$. Then
  $g = F(w) \cdot h / \Boxed{\sim_{F(A)}}$ for some $h \in \hom_\DD(F(A), F(B))$.
  Since $F$ is full there is an $f \in \hom_\CC(A, B)$ such that $F(f) = h$.
  Therefore, $g = F(w) \cdot F(f) / \Boxed{\sim_{F(A)}}
  = F(w \cdot f) / \Boxed{\sim_{F(A)}} = F(w \cdot f / \Boxed{\sim_A})$ by Claim~1.

  \medskip

  Let us now proceed with the proof.
  Take any coloring $\chi : \subobj{}{F(A)}{F(C)} \to k$ and define $\chi' : \subobj{}AC \to k$
  by $\chi'(f / \Boxed{\sim_A}) = \chi(F(f  / \Boxed{\sim_A})) = \chi(F(f) / \Boxed{\sim_{F(A)}})$
  by Claim~1. Since $C \overset\sim\longrightarrow (B)^A_{k, t}$ there is a $w \in \hom_\CC(B, C)$
  such that $|\chi'(w \cdot \subobj{}AB)| \le t$.
  But $\chi'(w \cdot \subobj{}AB) = \chi(F(w \cdot \subobj{}AB))
  = \chi(F(w) \cdot \subobj{}{F(A)}{F(B)})$ by Claim~2.
  Therefore, $|\chi(F(w) \cdot \subobj{}{F(A)}{F(B)})| \le t$.

  \medskip

  $(c)$ If $t_\CC(A) = \infty$ the statement is trivially true. Assume, therefore, that
  $t_\CC(A) = t \in \NN$ and let us show that $t_\DD(F(A)) \le t$.
  Take any $D \in \Ob(\DD)$ and $k \in \NN$.
  Since $F(\CC)$ is cofinal in $\DD$ there is a $B \in \Ob(\CC)$ such
  that $D \toDD F(B)$. From $t_\CC(A) = t$ it follows that there is a $C \in \Ob(\CC)$
  such that $C \longrightarrow (B)^A_{k, t}$. Then by $(a)$
  we have that $F(C) \longrightarrow (F(B))^{F(A)}_{k, t}$. Finally, $D \toDD F(B)$
  implies $F(C) \longrightarrow (D)^{F(A)}_{k, t}$ by Lemma~\ref{cerp.lem.easy}.
  This completes the proof that $t_\DD(F(A)) \le t$.

  \medskip

  $(d)$ The proof is analogous to $(c)$.
\end{proof}

\begin{LEM}\label{rament.lem.cofin-full-faith}
  Let $\CC$ and $\DD$ be locally small categories whose morphisms are mono, and let
  $F : \CC \to \DD$ be a full and faithful functor.
  
  $(a)$ $C \longrightarrow (B)^A_{k, t}$ if and only if $F(C) \longrightarrow (F(B))^{F(A)}_{k, t}$,
  for all $A, B, C \in \Ob(\CC)$ and $k, t \in \NN$.
  
  $(b)$ $C \overset\sim\longrightarrow (B)^A_{k, t}$ if and only if $F(C) \overset\sim\longrightarrow (F(B))^{F(A)}_{k, t}$,
  for all $A, B, C \in \Ob(\CC)$ and $k, t \in \NN$.

  $(c)$ If $F(\CC)$ is cofinal in $\DD$ then $t_\CC(A) = t_\DD(F(A))$
  and $\tilde t_\CC(A) = \tilde t_\DD(F(A))$ for all $A \in \Ob(\CC)$.
\end{LEM}
\begin{proof}
  Since $F$ is full and faithful, for each pair of objects $A, B \in \Ob(\CC)$ the functor $F$ induces a bijection
  $$
    F_{A,B} : \hom_{\CC}(A, B) \to \hom_{\DD}(F(A), F(B)) : f \mapsto F(f).
  $$
  The fact that $F$ is a functor immediately implies that
  $F_{A,A}(\Aut_{\CC}(A)) = \Aut_{\DD}(F(A))$ for all $A \in \Ob(\CC)$.
  Therefore, $F$ preserves automorphism groups. In particular, $|\Aut_\CC(A)| = |\Aut_\DD(F(A))|$
  for all $A \in \Ob(\CC)$.

  \medskip

  $(a)$ Implication $(\Rightarrow)$ was proved in Lemma~\ref{rament.lem.cofin-full-faith-weak}. Let us now prove $(\Leftarrow)$.
  Take any coloring $\chi : \hom_\CC(A, C) \to k$ and define
  $$
    \chi' : \hom_\DD(F(A), F(C)) \to k \text{ by } \chi'(F(f)) = \chi(f).
  $$
  Since $F(C) \longrightarrow (F(B))^{F(A)}_{k, t}$ there is an $F(w) \in \hom_\DD(F(B), F(C))$
  such that $|\chi'(F(w) \cdot \hom_\DD(F(A), F(B)))| \le t$. Note that
  $$
    F(w \cdot \hom_\CC(A, B)) = F(w) \cdot \hom_\DD(F(A), F(B))
  $$
  because $F$ is full and faithful. Therefore,
  \begin{align*}
    |\chi(w \cdot \hom_\CC(A, B))|  &= |\chi'(F(w \cdot \hom_\CC(A, B)))| \\
                                    &= |\chi'(F(w) \cdot \hom_\DD(F(A), F(B)))| \le t.
  \end{align*}

  \medskip

  $(b)$ Implication $(\Rightarrow)$ was proved in Lemma~\ref{rament.lem.cofin-full-faith-weak}. Let us now prove $(\Leftarrow)$.
  Take any coloring $\chi : \subobj{}{A}{C} \to k$ and having in mind Claims~1 and~2 from the proof of Lemma~\ref{rament.lem.cofin-full-faith-weak}
  define $\chi' : \subobj{}{F(A)}{F(C)} \to k$
  by $\chi'(F(f) / \Boxed{\sim_{F(A)}}) = \chi(f  / \Boxed{\sim_A})$.
  Since $F(C) \overset\sim\longrightarrow (F(B))^{F(A)}_{k, t}$ there is an $F(w) \in \hom_\CC(F(B), F(C))$
  such that $|\chi'(F(w) \cdot \subobj{}{F(A)}{F(B)})| \le t$. Note that
  \begin{align*}\textstyle
    \chi'(F(w) \cdot \subobj{}{F(A)}{F(B)})
    &\textstyle = \chi'(F(w \cdot \subobj{}{A}{B})) && \text{[Claim~2]}\\
    &\textstyle = \chi(w \cdot \subobj{}AB) && \text{[Claim~1]}.
  \end{align*}
  Therefore, $|\chi(w \cdot \subobj{}AB)| = |\chi'(F(w) \cdot \subobj{}{F(A)}{F(B)})| \le t$.

  \medskip

  $(c)$ The inequality $t_\DD(F(A)) \le t_\CC(A)$ was proved in Lemma~\ref{rament.lem.cofin-full-faith-weak}.
  Let us show that $t_\CC(A) \le t_\DD(F(A))$.
  If $t_\DD(F(A)) = \infty$ the statement is trivially true. Assume, therefore, that
  $t_\DD(F(A)) = t \in \NN$ and let us show that $t_\CC(A) \le t$.
  Take any $B \in \Ob(\CC)$ and $k \in \NN$. Then $t_\DD(F(A)) = t$ implies that
  there is a $D \in \Ob(\DD)$ such that $D \longrightarrow (F(B))^{F(A)}_{k, t}$.
  Since $F(\CC)$ is cofinal in $\DD$ there is a $C \in \Ob(\CC)$ such
  that $D \toDD F(C)$, so $F(C) \longrightarrow (F(B))^{F(A)}_{k, t}$ by Lemma~\ref{cerp.lem.easy}.
  Finally, $(a)$ gives us that $C \longrightarrow (B)^A_{k,t}$.
  This completes the proof that $t_\CC(A) \le t$.

  The proof of $\tilde t_\CC(A) = \tilde t_\DD(F(A))$ for all $A \in \Ob(\CC)$ follows by analogous
  arguments.
\end{proof}

\begin{COR}\label{crt.cor.cofinal}
  Let $\CC$ be a locally small category whose morphisms are mono, and let
  $\SS$ be a full cofinal subcategory of $\CC$. Then
  
  $(a)$ $t_\SS(A) = t_\CC(A)$ for all $A \in \Ob(\SS)$;

  $(b)$ $\tilde t_\SS(A) = \tilde t_\CC(A)$ for all $A \in \Ob(\SS)$;

  $(c)$ if $\CC$ is has finite small structural (embedding) Ramsey degrees
        then so does $\SS$;

  $(d)$ if $\CC$ is has the structural (embedding) Ramsey property
        then so does~$\SS$.
\end{COR}
\begin{proof}
  This is an immediate consequence of Lemma~\ref{rament.lem.cofin-full-faith}: just
  note that the inclusion functor $F : \SS \to \CC$ given by $F(A) = A$ on objects and
  $F(f) = f$ on morphisms is full and faithful and that $F(\SS)$ is cofinal in $\CC$.
\end{proof}

Let us now consider the case where $F : \CC \to \DD$ is a functor which is not necessarily full
and show a sufficient condition for transferring the embedding Ramsey phenomena from a category onto
its (not necessarily full!) subcategory. The criterion we present here bears resemblance to
the model theoretic notion of being existentially closed: we show that if a subcategory is
``existentially closed'' in its supercategory and the supercategory has some sort of embedding
Ramsey phenomenon then the same phenomenon is present in the subcategory. As a corollary we show that
if there is a faithful functor taking a cocomplete category $\CC$ into a category $\DD$
with some sort of embedding Ramsey phenomenon then $\CC$ also has the same Ramsey phenomenon.

Consider a finite, acyclic, bipartite digraph with loops 
where all the arrows go from one class of vertices into the other
and the out-degree of all the vertices in the first class is~2 (modulo loops):
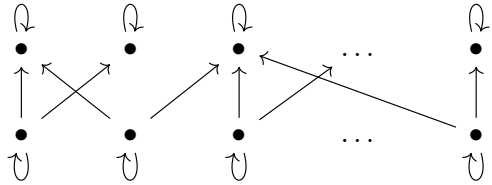
\begin{center}
  \begin{tikzcd}
    {\bullet} \arrow[loop above] & {\bullet} \arrow[loop above] & {\bullet} \arrow[loop above] & \ldots & {\bullet} \arrow[loop above] \\
    {\bullet} \arrow[loop below] \arrow[u] \arrow[ur] & {\bullet} \arrow[loop below] \arrow[ur] \arrow[ul] & {\bullet} \arrow[loop below] \arrow[u] \arrow[ur] & \ldots & {\bullet} \arrow[loop below] \arrow[u] \arrow[ull]
  \end{tikzcd}
\end{center}

\bigskip

Such a diagraph can be thought of as a category where the loops represent the identity morphisms, and will be referred to as
a \emph{binary category}. (Note that all the compositions in a binary category
are trivial since no nonidentity morphisms are composable.)

A \emph{binary diagram} in a category $\CC$ is a functor $F : \Delta \to \CC$ where $\Delta$ is a binary category,
$F$ takes the top row of $\Delta$ to the same object, and takes the bottom of $\Delta$ to the same object,
see Fig.~\ref{nrt.fig.2}. If $F$ takes the bottom row of the binary diagram $\Delta$ to an object $A$ and the top
row to an object $B$ then the diagram $F : \Delta \to \CC$ will be referred to as the \emph{$(A, B)$-diagram in $\CC$}.

\begin{figure}
\centering
\begin{tikzcd}
  {\bullet} & {\bullet} & {\bullet}
  & & B & B & B
\\
  {\bullet} \arrow[u] \arrow[ur] & {\bullet} \arrow[ur] \arrow[ul] & {\bullet} \arrow[ul] \arrow[u]
  & & A \arrow[u] \arrow[ur] & A \arrow[ur] \arrow[ul] & A \arrow[ul] \arrow[u]
\\
  & \Delta \arrow[rrrr, "F"]  & & & & \CC  
\end{tikzcd}
\caption{An $(A,B)$-diagram in $\CC$ (of shape $\Delta$)}
\label{nrt.fig.2}
\end{figure}
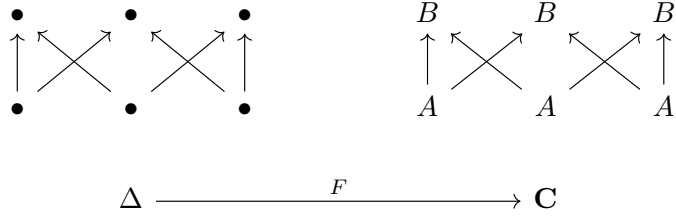

A \emph{walk} between two elements $x$ and $y$ of the top row of a binary category
consists of some vertices $x = t_0$, $t_1$, \dots, $t_k = y$ of the top row, some vertices
$b_1$, \dots, $b_k$ of the bottom row, and arrows $b_{j} \to t_{j-1}$ and $b_{j} \to t_{j}$, $1 \le j \le k$:
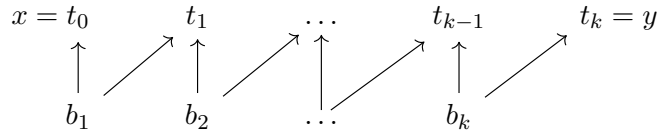
\begin{center}
  \begin{tikzcd}
    \llap{$x = \mathstrut$}t_0 & t_1 & \dots & t_{k-1} & t_k\rlap{$\mathstrut = y$} \\
    b_1 \arrow[u] \arrow[ur] & b_2 \arrow[u] \arrow[ur] & \dots \arrow[u] \arrow[ur] & b_k \arrow[u] \arrow[ur]
  \end{tikzcd}
\end{center}
A binary category is \emph{connected} if there is a walk between any pair of distinct vertices of the top row.
A \emph{connected component} of a binary category $\Delta$ is a maximal (with respect to inclusion) set $S$ of objects
of the top row such that there is a walk between any pair of distinct vertices from~$S$.

\begin{THM}\label{crt.thm.comp-cocone-general}
  Let $\CC$ and $\DD$ be locally small categories such that morphisms in both $\CC$ and $\DD$ are mono
  (and homsets are finite), and let $G : \DD \to \CC$ be a faithful functor.
  Assume that for any (finite) binary diagram $F : \Delta \to \DD$ in $\DD$ the following holds:
  if $GF : \Delta \to \CC$ has a compatible cocone in $\CC$ then $F : \Delta \to \DD$ has
  a compatible cocone in $\DD$. Then:

  $(a)$ $t_\DD(A) \le t_\CC(G(A))$ for all $A \in \Ob(\DD)$;

  $(b)$ if $\CC$ has finite small embedding Ramsey degrees then so does $\DD$;

  $(c)$ if $\CC$ has the embedding Ramsey property then so does~$\DD$.
  \end{THM}
\begin{proof}
  $(a)$
  Take any $A \in \Ob(\DD)$ and assume that $t_\CC(G(A)) = t \in \NN$. To show that $t_\DD(A) \le t$
  take any $k \in \NN$ and $B \in \Ob(\DD)$ such that $A \toDD B$. Since $t_\CC(G(A)) = t$
  there is a $C \in \Ob(\CC)$ such that $C \longrightarrow (G(B))^{G(A)}_{k, t}$.
  
  Let us now construct a binary diagram in $\DD$ as follows. Let
  $$
    \hom_\CC(G(B), C) = \{e_i : i \in I\}.
  $$
  (In case of $\CC$ and $\DD$ having finite homsets the index set $I$ will be finite and then the diagram that we
  construct will also be finite; all the other elements of the proof remain unchanged.)
  Intuitively, for each $e_i \in \hom_\CC(G(B), C)$ we add a copy of $B$ to the diagram, and whenever $e_i \cdot G(u) = e_j \cdot G(v)$
  in $\CC$ for some $u, v \in \hom_\DD(A, B)$ we add a copy of A to the diagram together with two arrows: one going into the $i$th copy of $B$
  labelled by $u$ and another one going into the $j$th copy of $B$ labelled by~$v$:
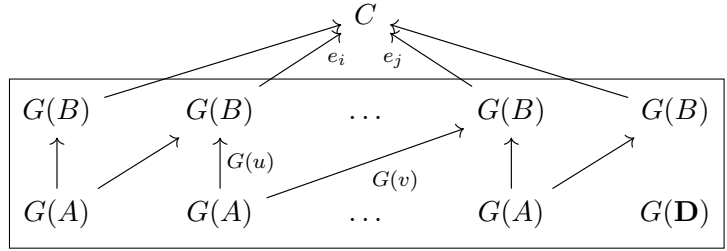
  \begin{center}
    \begin{tikzcd}[execute at end picture={
                       \draw (-4.75,-1.75) rectangle (4.75,0.5);
                  }]
      & & C
    \\
      G(B) \arrow[urr] & G(B) \arrow[ur, near end, "e_i"'] & \ldots & G(B) \arrow[ul, near end, "e_j"] & G(B) \arrow[ull]
    \\
      G(A) \arrow[u] \arrow[ur] & G(A) \arrow[urr, "G(v)"'] \arrow[u, "G(u)"'] & \ldots & G(A) \arrow[ur] \arrow[u] & G(\DD)
    \end{tikzcd}
  \end{center}
  Note that, by the construction, the diagram $GF : \Delta \to \CC$ has a compatible cocone in $\CC$.

  Formally, let $\Delta$ be the binary category whose objects are
  $$
    \Ob(\Delta) = I \union \{(u, v, i, j) : i, j \in I; 
                                            u, v \in \hom_\DD(A, B); e_i \cdot G(u) = e_j \cdot G(v)\}
  $$
  and whose nonidentity arrows are
  $$
    \hom_\Delta((u, v, i, j), i) = \{u\} \text{\quad and\quad} \hom_\Delta((u, v, i, j), j) = \{v\}.
  $$
  Let $F : \Delta \to \DD$ be the following diagram whose action on objects is:
  $$
    F(i) = B \text{\quad and\quad} F((u, v, i, j)) = A
  $$
  for all $i, (u,v,i,j) \in \Ob(\Delta)$, and whose action on nonidentity morphisms is $F(g) = g$:
  \begin{center}
    \begin{tikzcd}[column sep=small]
    i &  & j
    & & B &  & B
  \\
      & (u, v, i, j) \arrow[ur, "v"'] \arrow[ul, "u"] &  
    & &   & A \arrow[ur, "v"'] \arrow[ul, "u"] & 
  \\
    & \Delta \arrow[rrrr, "F"]  & & & & \DD
    \end{tikzcd}
  \end{center}

  As we have already observed in the informal discussion above, the diagram $GF : \Delta \to \CC$ has a compatible cocone in $\CC$,
  so, by the assumption, the same holds for $F$ in~$\DD$. Therefore, there is
  a $D \in \Ob(\DD)$ and morphisms $f_i : B \to D$, $i \in I$, such that the following diagram in $\DD$ commutes:
  \begin{center}
    \begin{tikzcd}
        & & D
      \\
        B \arrow[urr] & B \arrow[ur, "f_i"'] & \ldots & B \arrow[ul,"f_j"] & B \arrow[ull]
      \\
        A \arrow[u] \arrow[ur] & A \arrow[urr,"v"] \arrow[u, "u"'] & \ldots & A \arrow[ur] \arrow[u] & 
    \end{tikzcd}
  \end{center}
  Let us show that in $\DD$ we have $D \longrightarrow (B)^A_{k, t}$. Take any $k$-coloring
  $$
    \hom_\DD(A, D) = X_1 \union \ldots \union X_k,
  $$
  and define a coloring
  $$
    \hom_\CC(G(A), C) = X'_1 \union \ldots \union X'_k
  $$
  as follows. For $j \in \{2, \ldots, k\}$ let
  $$
    X'_j = \{e_p \cdot G(u) : p \in I, u \in \hom_\DD(A, B), f_p \cdot u \in X_j \},
  $$
  and then let
  $
    X'_1 = \hom_\CC(G(A), C) \setminus \bigcup_{j=2}^k X'_j
  $.
  
  Let us show that $X'_1 \cup \ldots \cup X'_k$ is a coloring of $\hom_\CC(G(A), C)$, i.e.\ that
  $X'_i \sec X'_j = \0$ whenever $i \ne j$. By definition of $X'_1$
  it suffices to consider the case where $i \ge 2$ and $j \ge 2$. 
  Assume, to the contrary, that there is an $h \in X'_i \sec X'_j$ for some $i \ne j$, $i \ge 2$, $j \ge 2$.
  Then $h = e_p \cdot G(u)$ for some $p \in I$ and some $u \in \hom_\DD(A, B)$ such that $f_p \cdot u \in X_i$, and
  $h = e_q \cdot G(v)$ for some $q \in I$ and some $v \in \hom_\DD(A, B)$ such that $f_q \cdot v \in X_j$.
  Then
  $$
    e_p \cdot G(u) = h = e_q \cdot G(v),
  $$
  so, by definition of $\Delta$, we have that $(u, v, p, q) \in \Ob(\Delta)$.
  Consequently,
  $f_p \cdot u = f_q \cdot v$ because $D$ and morphisms $f_i : B \to D$, $i \in I$, constitute a compatible cocone for
  the diagram $F : \Delta \to \DD$ in~$\DD$. Therefore, $f_p \cdot u = f_q \cdot v \in X_i \sec X_j$ -- contradiction.
  
  Since $C \longrightarrow (G(B))^{G(A)}_{k, t}$, there is an $e_\ell \in \hom_\CC(G(B), C)$ and
  $j_1, \ldots, j_t \in \{1, \ldots, k\}$ such that
  \begin{equation}\label{crt.eq.bin-cat-t}
    e_\ell \cdot \hom_\CC(G(A), G(B)) \subseteq \bigcup_{m = 1}^t X'_{j_m}
  \end{equation}
  Let us show that $f_\ell \cdot \hom_\DD(A, B) \subseteq \bigcup_{m = 1}^t X_{j_m}$.
  Take any $u \in \hom_\DD(A, B)$. Then, because of \eqref{crt.eq.bin-cat-t}, there is an $s \in \{1, \ldots, t\}$ such that
  $e_\ell \cdot G(u) \in X'_{j_s}$. Let us show that $f_\ell \cdot u \in X_{j_s}$.
  If $j_s \ge 2$ then by definition of $X'_{j_s}$ we have that $f_\ell \cdot u \in X_{j_s}$.
  Assume, now, that $j_s = 1$ and suppose that $f_\ell \cdot u \notin X_1$. Then $f_\ell \cdot u \in X_n$
  for some $n \ge 2$. But then $e_\ell \cdot G(u) \in X'_n$. On the other hand, $e_\ell \cdot G(u) \in X'_1$ by assumption ($j_s = 1$),
  so $n \ne 1$ and $X'_1 \cap X'_n \ni e_\ell \cdot G(u)$ -- contradiction.
  Therefore, $f_\ell \cdot u \in X_{j_s}$ proving, thus, that $f_\ell \cdot \hom_\DD(A, B) \subseteq \bigcup_{m = 1}^t X_{j_m}$.

  \medskip

  $(b)$ and $(c)$ are immediate consequences of $(a)$.
\end{proof}

\begin{COR}\label{nrt.thm->cor.1}\cite{masul-drp}
  Let $\CC$ be a locally small category whose morphisms are mono (and homsets are finite)
  and let $\SS$ be a subcategory of $\CC$.
  Assume that for any (finite) binary diagram $F : \Delta \to \SS$ the following holds:
  if $F$ has a compatible cocone in $\CC$ then $F$ has
  a compatible cocone in $\SS$. Then:
  
  $(a)$ $t_\SS(A) \le t_\CC(A)$ for all $A \in \Ob(\SS)$;

  $(b)$ if $\CC$ has finite small embedding Ramsey degrees then so does $\SS$;
  
  $(c)$ if $\CC$ has the embedding Ramsey property then so does $\SS$.
\end{COR}
\begin{proof}
  Just take $G : \SS \to \CC$ to be the inclusion functor
  given by $G(A) = A$ on objects and $G(f) = f$ on morphisms and apply
  Theorem~\ref{crt.thm.comp-cocone-general}.
\end{proof}

We say that a category $\DD$ has \emph{amalgamation of (finite) binary diagrams} if every
(finite) binary diagram $F : \Delta \to \DD$ has a compatible cocone in $\DD$.

\begin{COR}\label{crt.cor.comp-cocone-amalg-closed}
  Let $\DD$ be locally small category (category with finite homsets) such that morphisms in $\DD$ are mono
  and assume that $\DD$ has amalgamation of (finite) binary diagrams.
  
  $(a)$ If there is a faithful functor $F : \DD \to \CC$ from $\DD$ into a locally small category
  (category with finite homsets) $\CC$ which has small embedding Ramsey degrees, then $\DD$ has small embedding Ramsey degrees
  and $t_\DD(A) \le t_\CC(F(A))$ for all $A \in \Ob(\DD)$.
  
  $(b)$ If there is a faithful functor from $\DD$ into a locally small category
  (category with finite homsets) $\CC$ which has the embedding Ramsey property, then $\DD$ has the embedding Ramsey property.

  $(c)$ If $\DD$ is a subcategory of a locally small category
  (category with finite homsets) $\CC$ which has small embedding Ramsey degrees, then $\DD$ has small embedding Ramsey degrees
  and $t_\DD(A) \le t_\CC(A)$ for all $A \in \Ob(\DD)$.
  
  $(d)$ If $\DD$ is a subcategory of a locally small category
  (category with finite homsets) $\CC$ which has the embedding Ramsey property, then $\DD$ has the embedding Ramsey property.
\end{COR}

\section{Products and pullbacks of categories}
\label{rpppg.sec.prod-pb}

Having taken care of the behavior of Ramsey degrees in subcategories,
we shall now turn to products and pullbacks of categories.
We start by showing that Ramsey degrees are multiplicative.

\begin{THM}\label{rament.thm.srd-mult}
  Let $\CC_1$ and $\CC_2$ be categories whose morphisms are mono and homsets are finite. Then
  for all $A_1 \in \Ob(\CC_1)$ and $A_2 \in \Ob(\CC_2)$:
  $$
    t_{\CC_1 \times \CC_2}(A_1, A_2) = t_{\CC_1}(A_1) \cdot t_{\CC_2}(A_2),
  $$
  and this holds even in case some of the above degrees are infinite
  (where we take $\infty \cdot \infty = \infty$). Consequently,
  $$
   \tilde t_{\CC_1 \times \CC_2}(A_1, A_2) = \tilde t_{\CC_1}(A_1) \cdot \tilde t_{\CC_2}(A_2).
  $$
\end{THM}
\begin{proof}
  The second part of the statement is an immediate consequence of the first part of the statement.
  Since homsets in both $\CC_1$ and $\CC_2$ are finite, the homsets in $\CC_1 \times \CC_2$ are also finite,
  whence follows that the automorphism groups in $\CC_1$, $\CC_2$ and $\CC_1 \times \CC_2$ are finite.
  Therefore, the second part of the statement follows from the first part of the statement,
  Proposition~\ref{rdbas.prop.sml} and the fact that
    $$
      |\Aut_{\CC_1 \times \CC_2}(A_1, A_2)| = |\Aut_{\CC_1}(A_1)| \cdot |\Aut_{\CC_2}(A_2)|.
    $$

  To show the first part of the statement take any $A_1 \in \Ob(\CC_1)$ and $A_2 \in \Ob(\CC_2)$.
  We have already proved in \cite[Theorem 3.3]{masul-kpt} that $t_{\CC_1 \times \CC_2}(A_1, A_2) \le t_{\CC_1}(A_1) \cdot t_{\CC_2}(A_2)$.
    In order to complete the proof we still have to show that
    $t_{\CC_1 \times \CC_2}(A_1, A_2) \ge t_{\CC_1}(A_1) \cdot t_{\CC_2}(A_2)$.
    Note that this is trivially true in case $t_{\CC_1 \times \CC_2}(A_1, A_2) = \infty$.
    Therefore, we now consider the case when $t_{\CC_1 \times \CC_2}(A_1, A_2) < \infty$.
  
    \bigskip
  
    Step 1. If $t_{\CC_1 \times \CC_2}(A_1, A_2) < \infty$ then $t_{\CC_1}(A_1) < \infty$ and $t_{\CC_2}(A_2) < \infty$.
    
    \medskip
    
    Proof. Suppose that $t_{\CC_1}(A_1) = \infty$ and let us show that $t_{\CC_1 \times \CC_2}(A_1, A_2) = \infty$, that is,
    $t_{\CC_1 \times \CC_2}(A_1, A_2) \ge n$ for all $n \in \NN$.
    
    Take any $n \in \NN$. Then $t_{\CC_1}(A_1) \ge n$ (because $t_{\CC_1}(A_1) = \infty$ by assumption) and $t_{\CC_2}(A_2) \ge 1$
    (trivially). Since $t_{\CC_1}(A_1) \ge n$ there is a $k_1 \in \NN$ and a $B_1 \in \Ob(\CC_1)$ such that
    \begin{align}
      (\forall &C_1 \in \Ob(\CC_1))(\exists \chi_1 : \hom_{\CC_1}(A_1, C_1) \to k_1) \nonumber\\
               &(\forall w_1 \in \hom_{\CC_1}(B_1, C_1)) \; |\chi_1(w_1 \cdot \hom_{\CC_1}(A_1, B_1))| \ge n. \label{rament.eq.x-1}
    \end{align}
    On the other hand, $t_{\CC_2}(A_2) \ge 1$ implies that there is a $k_2 \in \NN$ and a $B_2 \in \Ob(\CC_2)$ such that
    \begin{align}
      (\forall &C_2 \in \Ob(\CC_2))(\exists \chi_2 : \hom_{\CC_2}(A_2, C_2) \to k_2) \nonumber\\
               &(\forall w_2 \in \hom_{\CC_2}(B_2, C_2)) \; |\chi_2(w_2 \cdot \hom_{\CC_2}(A_2, B_2))| \ge 1. \label{rament.eq.x-2}
    \end{align}
    We are going to show that $k_1 \cdot k_2$ and $(B_1, B_2) \in \Ob(\CC_1 \times \CC_2)$ are the parameters we are looking
    for to show that $t_{\CC_1 \times \CC_2}(A_1, A_2) \ge n$.
  
    Take any $(C_1, C_2) \in \Ob(\CC_1 \times \CC_2)$ and choose $\chi_1 : \hom_{\CC_1}(A_1, C_1) \to k_1$
    whose existence is guaranteed by \eqref{rament.eq.x-1}, and $\chi_2 : \hom_{\CC_2}(A_2, C_2) \to k_2$
    whose existence is guaranteed by \eqref{rament.eq.x-2}. Let 
    $$
      \chi : \hom_{\CC_1 \times \CC_2}((A_1, A_2), (C_1, C_2)) \to k_1 \times k_2
    $$
    be the coloring defined by
    $$
      \chi(f_1, f_2) = (\chi_1(f_1), \chi_2(f_2))
    $$
    where we implicitly used the fact that
    $$
      \hom_{\CC_1 \times \CC_2}((A_1, A_2), (C_1, C_2)) = \hom_{\CC_1}(A_1, C_1) \times \hom_{\CC_2}(A_2, C_2).
    $$
    Finally, take any $(w_1, w_2) \in \hom_{\CC_1 \times \CC_2}((B_1, B_2), (C_1, C_2))$ and note that
    $$
      |\chi((w_1, w_2) \cdot \hom_{\CC_1 \times \CC_2}((A_1, A_2), (B_1, B_2)))| \ge n
    $$
    because
    \begin{align*}
     |\chi(&(w_1, w_2) \cdot \hom_{\CC_1 \times \CC_2}((A_1, A_2), (B_1, B_2)))| =\\
           & = |\chi_1(w_1 \cdot \hom_{\CC_1}(A_1, B_1)) \times \chi_2(w_2 \cdot \hom_{\CC_2}(A_2, B_2))| \ge\\
           & \ge |\chi_1(w_1 \cdot \hom_{\CC_1}(A_1, B_1))| \ge n.
    \end{align*}
    This concludes the proof of Step 1.
    
    \bigskip
    
    Step 2. If $t_{\CC_1 \times \CC_2}(A_1, A_2) = n$, $t_{\CC_1}(A_1) = p$ and $t_{\CC_2}(A_2) = q$ where
    $n, p, q \in \NN$ then $p \le \lfloor n/q \rfloor$, whence $n \ge p \cdot q$.
    
    \medskip
    
    Proof. Let us show that $t_{\CC_1}(A_1) \le \lfloor n/q \rfloor$. Take any
    $k_1 \in \NN$ and $B_1 \in \Ob(\CC_1)$.
    Since $t_{\CC_2}(A_2) \ge q$, there is a $k_2 \in \NN$ and $B_2 \in \Ob(\CC_2)$ such that
    \begin{align}
      (\forall &C_2 \in \Ob(\CC_2))(\exists \chi_2 : \hom_{\CC_2}(A_2, C_2) \to k_2) \nonumber\\
               &(\forall w_2 \in \hom_{\CC_2}(B_2, C_2)) \; |\chi_2(w_2 \cdot \hom_{\CC_2}(A_2, B_2))| \ge q. \label{rament.eq.x-3}
    \end{align}
    On the other hand, $t_{\CC_1 \times \CC_2}(A_1, A_2) \le n$, so for $(B_1, B_2)$ that we have just chosen
    and $k_1 \times k_2$ as the set of colors there is a $(C_1, C_2) \in \Ob(\CC_1 \times \CC_2)$ such that
    \begin{align}
      (\forall &\chi : \hom_{\CC_1 \times \CC_2}((A_1, A_2), (C_1, C_2)) \to k_1 \times k_2) \nonumber\\
               &(\exists (w_1, w_2) \in \hom_{\CC_1 \times \CC_2}((B_1, B_2), (C_1, C_2))) \nonumber \\
               &|\chi((w_1, w_2) \cdot \hom_{\CC_1 \times \CC_2}((A_1, A_2), (B_1, B_2)))| \le n. \label{rament.eq.x-4}
    \end{align}
    Let us show that $C_1 \longrightarrow (B_1)^{A_1}_{k_1, \lfloor n/q \rfloor}$.
    Take any $\chi_1 : \hom_{\CC_1}(A_1, C_1) \to k_1$. Having in mind~\eqref{rament.eq.x-3},
    for $C_2$ we have obtained in the previous paragraph there is a
    $\chi_2 : \hom_{\CC_2}(A_2, C_2) \to k_2$ satisfying
    \begin{equation}
      (\forall w_2 \in \hom_{\CC_2}(B_2, C_2)) \; |\chi_2(w_2 \cdot \hom_{\CC_2}(A_2, B_2))| \ge q.
      \label{rament.eq.x-5}
    \end{equation}
    Define
    $$
      \chi : \hom_{\CC_1 \times \CC_2}((A_1, A_2), (C_1, C_2)) \to k_1 \times k_2
    $$
    by
    $$
      \chi(f_1, f_2) = (\chi_1(f_1), \chi_2(f_2))
    $$
    where, as above, we implicitly used the fact that
    $$
      \hom_{\CC_1 \times \CC_2}((A_1, A_2), (C_1, C_2)) = \hom_{\CC_1}(A_1, C_1) \times \hom_{\CC_2}(A_2, C_2).
    $$
    Because of \eqref{rament.eq.x-4} there is a $(w_1, w_2) \in \hom_{\CC_1 \times \CC_2}((B_1, B_2), (C_1, C_2))$ such that
    $$
      |\chi((w_1, w_2) \cdot \hom_{\CC_1 \times \CC_2}((A_1, A_2), (B_1, B_2)))| \le n.
    $$
    Then
    \begin{align*}
     n \ge |\chi(&(w_1, w_2) \cdot \hom_{\CC_1 \times \CC_2}((A_1, A_2), (B_1, B_2)))| =\\
           & = |\chi_1(w_1 \cdot \hom_{\CC_1}(A_1, B_1)) \times \chi_2(w_2 \cdot \hom_{\CC_2}(A_2, B_2))| =\\
           & = |\chi_1(w_1 \cdot \hom_{\CC_1}(A_1, B_1))| \cdot |\chi_2(w_2 \cdot \hom_{\CC_2}(A_2, B_2))| \ge\\
           & = |\chi_1(w_1 \cdot \hom_{\CC_1}(A_1, B_1))| \cdot q
    \end{align*}
    because of \eqref{rament.eq.x-5}. Therefore,
    $$
      |\chi_1(w_1 \cdot \hom_{\CC_1}(A_1, B_1))| \le \lfloor n/q \rfloor.
    $$
    This concludes the proof of Step~2 and the proof of the theorem.
  \end{proof}

Let us now move on to the proof of the generalization of Theorem~\ref{rpppg.thm.bodirsky-1}.
In analogy to~\cite{KPT} we shall say that
a functor $F : \CC \to \DD$ is \emph{reasonable} if for every $C \in \Ob(\CC)$, every $B \in \Ob(\DD)$
and every $h \in \hom_\DD(F(C), B)$ there is a $D \in \Ob(\CC)$ and a $g \in \hom_\CC(C, D)$ such that
$F(D) = B$ and $F(g) = h$:
\begin{center}
  \begin{tikzcd}
    C \arrow[r, "g"] \arrow[d, dashed, mapsto, "F"'] & D \arrow[d, dashed, mapsto, "F"] & & \CC \arrow[d, "F"]\\
    F(C) \arrow[r, "h"'] & B & & \DD
  \end{tikzcd}
\end{center}

\begin{THM}\label{rament.thm.pullback}
  Let $\CC_1$ and $\CC_2$ be categories whose morphisms are mono and homsets are finite,
  let $\DD$ be a directed category and let $F_1 : \CC_1 \to \DD$ and $F_2 : \CC_2 \to \DD$
  be reasonable functors. Let $\PP$ be the pullback of
  $\CC_1 \overset{F_1}\longrightarrow \DD \overset{F_2}\longleftarrow \CC_2$. Then
  the following holds (with Convention~($\dagger$) in mind):
  
  $(a)$ $t_{\PP}(A_1, A_2) = t_{\CC_1}(A_1) \cdot t_{\CC_2}(A_2)$ for all $(A_1, A_2) \in \Ob(\PP)$;

  $(b)$ $\tilde t_{\PP}(A_1, A_2) = \tilde t_{\CC_1}(A_1) \cdot \tilde t_{\CC_2}(A_2)$ for all $(A_1, A_2) \in \Ob(\PP)$;

  $(c)$ if both $\CC_1$ and $\CC_2$ have finite small structural (embedding) Ramsey degrees then so does~$\PP$;

  $(d)$ if both $\CC_1$ and $\CC_2$ have the structural (embedding) Ramsey property then so does~$\PP$.
\end{THM}
\begin{proof}
  $(a)$ and $(b)$.
  Take $\PP$ to be the full subcategory of $\CC_1 \times \CC_2$ spanned by $(C_1, C_2)$ such that
  $F_1(C_1) = F_2(C_2)$ and let us show that $\PP$ is a cofinal subcategory of $\CC_1 \times \CC_2$.
  Take any $(C_1, C_2) \in \Ob(\CC_1 \times \CC_2)$. Since $\DD$ is directed, there is a $B \in \Ob(\BB)$
  such that $F_1(C_1) \toDD B$ and $F_2(C_2) \toDD B$. Take any $h_1 \in \hom_\DD(F_1(C_1), B)$ and
  $h_2 \in \hom_\DD(F_2(C_2), B)$. Since both $F_1$ and $F_2$ are reasonable, there exist
  $D_1 \in \Ob(\CC_1)$, $D_2 \in \Ob(\CC_2)$, $g_1 \in \hom_\CC(C_1, D_1)$
  and $g_2 \in \hom_\CC(C_2, D_2)$ such that $F_1(D_1) = B$, $F_2(D_2) = B$,
  $F_1(g_1) = h_1$ and $F_2(g_2) = h_2$. Now, $(D_1, D_2) \in \Ob(\PP)$ because
  $F_1(D_1) = F_2(D_2) = B$ and $(C_1, C_2) \toCCIII (D_1, D_2)$ because
  $(g_1, g_2) \in \hom_{\CC_1 \times \CC_2}((C_1, C_2), (D_1, D_2))$.
  
  Therefore, $\PP$ is a full cofinal subcategory of $\CC_1 \times \CC_2$.
  Take any $(A_1, A_2) \in \Ob(\PP)$. Corollary~\ref{crt.cor.cofinal}
  and Theorem~\ref{rament.thm.srd-mult} now yield
  \begin{align*}
    t_\PP(A_1, A_2) = t_{\CC_1 \times \CC_2}(A_1, A_2) &= t_{\CC_1}(A_1) \cdot t_{\CC_2}(A_2) \text{ and}\\
    \tilde t_\PP(A_1, A_2) = \tilde t_{\CC_1 \times \CC_2}(A_1, A_2) &= \tilde t_{\CC_1}(A_1) \cdot \tilde t_{\CC_2}(A_2).
  \end{align*}

  $(c)$ and $(d)$ follow directly from $(a)$ and $(b)$.
\end{proof}

Let us now relate the result above to a powerful result of M.\ Bodirsky about string amalgamations classes with
the Ramsey property. Let $\Theta$ be a first-order signature. A class $\KK$ of $\Theta$-structures
is a \emph{strong amalgamation class} if for all $\calA, \calB_1, \calB_2 \in \KK$ and
all embeddings $f_1 : \calA \hookrightarrow \calB_1$ and $f_2 : \calA \hookrightarrow \calB_2$
there is a $\calC \in \KK$ and embeddings $g_1 : \calB_1 \hookrightarrow \calC$ and $g_2 : \calB_2 \hookrightarrow \calC$
such that $g_1 \circ f_1 = g_2 \circ f_2$ and $g_1(B_1) \cap g_2(B_2) = g_1(f_1(A)) = g_2(f_2(A))$, where
$A$, $B_1$, $B_2$ and $C$ are the underlying sets of $\calA$, $\calB_1$, $\calB_2$ and $\calC$, respectively:
\begin{center}
  \begin{tikzcd}
    \calB_2 \arrow[r, "g_2"] & \calC \\
    \calA \arrow[u, "f_2"] \arrow[r, "f_1"'] & \calB_1 \arrow[u, "g_1"']
  \end{tikzcd}
\end{center}
This means that in the amalgam $\calC$ the images of $\calB_1$ and $\calB_2$ under $g_1$ and $g_2$ respectively
overlap only where necessary, that is, only in the image of $\calA$ under $g_1 \circ f_1 = g_2 \circ f_2$.

Let $\Theta_1$ and $\Theta_2$ be disjoint first-order signatures, let $\KK_1$ be a class of finite $\Theta_1$-structures
and $\KK_2$ a class of finite $\Theta_2$-structures. Then $\KK_1 \otimes \KK_2$ is a class of $\Theta_1 \union \Theta_2$
structures defined as follows: a $(\Theta_1 \union \Theta_2)$-structure
$\calA = (A, \Theta_1^\calA, \Theta_2^\calA)$ belongs to $\KK_1 \otimes \KK_2$ if and only if
$(A, \Theta_1^\calA) \in \KK_1$ and $(A, \Theta_2^\calA) \in \KK_2$.
Bodirsky's statement now reads as follows:

\begin{THM}[Free superposition of Ramsey classes]\cite[Theorem 1.3]{Bodirsky}\label{rpppg.thm.bodirsky-1}
  Let $\KK_1$ and $\KK_2$ be strong amalgamation classes of finite structures
  in disjoint finite relational signatures $\Theta_1$ and $\Theta_2$, respectively, and with
  the Ramsey property. Then $\KK_1 \otimes \KK_2$ has the Ramsey property.
\end{THM}

The original proof given in~\cite{Bodirsky} uses results of model theory and topological dynamics,
and in in some instances heavily rely on \cite{KPT} which connects topological
dynamics and Ramsey theory. This statements was later generalized by M.\ Soki\'c
who provided purely combinatorial proof in~\cite[Corollary~2]{sokic-directed-graphs-boron-trees}.

It is important to note that in case $\KK_1$ and $\KK_2$ are classes of first-order structures,
$\KK_1 \otimes \KK_2$ \emph{is not} the pullback $\PP$ of $\KK_1 \overset{U_1}\longrightarrow \Set \overset{U_2}\longleftarrow \KK_2$
since there are much more morphisms in $\PP$ than there
are in $\KK_1 \otimes \KK_2$. Using the strategies we have developed in this paper
we can show the following ``sideways generalization'' of Theorem~\ref{rpppg.thm.bodirsky-1}
where we can generalize the above statement to arbitrary first-order languages at the cost of
replacing the strong amalgamation requirement by a more demanding one.

Let $\Theta$ be a first-order signature and let $\KK$ be a class of finite $\Theta$-structures.
We say that a class $\KK$ is \emph{reasonable}~\cite{KPT} if for every $\calA = (A, \Theta^\calA)$ in $\KK$
and every injective map $f : A \to B$ where $B$ is a finite set there is a structure
$\calB = (B, \Theta^\calB)$ in $\KK$ such that $f$ is an embedding $\calA \hookrightarrow \calB$.

\begin{COR}\label{crt.cor.pullback}
  Let $\KK_1$ and $\KK_2$ be reasonable classes of finite structures
  in first-order signatures $\Theta_1$ and $\Theta_2$, respectively,
  and assume that $\KK_1 \otimes \KK_2$ has amalgamation of finite binary diagrams.

  $(a)$ For every $\calA = (A, \Theta_1^\calA, \Theta_2^\calA)$ in $\KK_1 \otimes \KK_2$ we have that
  (with Convention~($\dagger$) in mind):
  \begin{align*}
    t_{\KK_1 \otimes \KK_2}(A, \Theta_1^\calA, \Theta_2^\calA) &\le t_{\KK_1}(A, \Theta_1^\calA) \cdot t_{\KK_2}(A, \Theta_2^\calA) \text{ and}\\
    \tilde t_{\KK_1 \otimes \KK_2}(A, \Theta_1^\calA, \Theta_2^\calA) &\le \tilde t_{\KK_1}(A, \Theta_1^\calA) \cdot \tilde t_{\KK_2}(A, \Theta_2^\calA).
  \end{align*}

  $(b)$ If both $\KK_1$ and $\KK_2$ have finite structural (embedding) Ramsey degrees then so does $\KK_1 \otimes \KK_2$.

  $(c)$ If both $\KK_1$ and $\KK_2$ have the structural (embedding) Ramsey property then so does $\KK_1 \otimes \KK_2$.
\end{COR}
\begin{proof}
  Note that $\KK_1$ and $\KK_2$ are categories of finite structures, where we take embeddings as morphisms.
  The obvious forgetful functors $U_i : \KK_i \to \Set : (A, \Theta^\calA) \mapsto A : f \mapsto f$, $i \in \{1, 2\}$,
  are reasonable because $\KK_1$ and $\KK_2$ are reasonable classes. The category $\Set$ is obviously directed.
  Note that $\KK_1 \otimes \KK_2$ is \emph{not} isomorphic to the pullback $\PP$
  of $\KK_1 \overset{U_1}\longrightarrow \Set \overset{U_2}\longleftarrow \KK_2$.
  However, there is a faithful functor $F : \KK_1 \otimes \KK_2 \to \PP$ given by $F(A, \Theta_1^\calA, \Theta_2^\calA)
  = (A, \Theta_1^\calA, A, \Theta_2^\calA)$ on objects and by $F(f) = (f, f)$ on morphisms.
  Therefore, Corollary~\ref{crt.cor.comp-cocone-amalg-closed} and Theorem~\ref{rament.thm.pullback} yield
  $$
  t_{\KK_1 \otimes \KK_2}(A, \Theta_1^\calA, \Theta_2^\calA) \le
  t_\PP(A, \Theta_1^\calA, A, \Theta_2^\calA)
  = t_{\KK_1}(A, \Theta_1^\calA) \cdot t_{\KK_2}(A, \Theta_2^\calA).
  $$
  The analogous statement for $\tilde t$ now follows immediately.
\end{proof}

\section{The Grothendieck construction}
\label{rpppg.sec.grothendieck}

Let $\Theta$ be a relational signature and let $x_1, \ldots, x_n$ be elements of $\calF$.
Let us take a closer look into the finitely generated substructures of $(\calF, x_1, \ldots, x_n)$. If
$(\calA, a_1, \ldots, a_n)$ and $(\calB, b_1, \ldots, b_n)$ are finitely generated
structures that embed into $(\calF, x_1, \ldots, x_n)$
then $\{a_1, \ldots, a_n\}$ generates a substructure of $\calA$ which is isomorphic to the substructure of
$\calF$ generated by $\{x_1, \ldots, x_n\}$. Of course, the same holds for $\calB$.
Let $X = \{x_1, \ldots, x_n\}$ and let $\calX = \calF[X]$ be the substructure of $\calF$ generated by $X$.
Clearly, there are unique embeddings
\begin{align*}
  f_\calA &: (\calX, x_1, \ldots, x_n) \hookrightarrow (\calA, a_1, \ldots, a_n) \text{ and }\\
  f_\calB &: (\calX, x_1, \ldots, x_n) \hookrightarrow (\calB, b_1, \ldots, b_n),
\end{align*}
and for every embedding $h : (\calA, a_1, \ldots, a_n) \hookrightarrow (\calB, b_1, \ldots, b_n)$
we have that
\begin{center}
  \begin{tikzcd}
    & {(\calX, x_1, \ldots, x_n)} \arrow[dl, hook, "f_\calA"'] \arrow[dr, hook, "f_\calB"] & \\
    {(\calA, a_1, \ldots, a_n)} \arrow[rr, hook, "h"']  & & {(\calB, b_1, \ldots, b_n)} 
  \end{tikzcd}
\end{center}
It is now obvious that what we are looking at is, actually, a statement about transporting the Ramsey property
from a category onto a slice category.
As slice categories are special cases of the Grothendieck construction, we shall now prove a general statement about
transporting the Ramsey property from a category $\CC$ onto the Grothendieck category $\GG(\CC, F)$.
Let us now fix the terminology concerning the slice and Grothendieck categories.

Let $\CC$ be a locally small category and $X \in \Ob(\CC)$.
A \emph{slice category} $\slice X\CC$ is a category whose objects are pairs $(f_A, A)$ where $f_A  \in \hom_\CC(X, A)$.
A morphism $(f_A, A) \to (f_B, B)$ in $\slice X\CC$ is every morphism $h : A \to B$ such that $f_B = h \cdot f_A$.

For a locally small category $\CC$ and a functor $H : \CC \to \Set$ let $\GG(\CC, H)$ denote the category whose
objects are pairs $(C, x)$ where $C \in \Ob(\CC)$ and $x \in H(C)$,
morphisms are of the form $f : (C, x) \to (D, y)$ where $f \in \hom_\CC(C, D)$ and $H(f)(x) = y$, and
the composition of morphisms is as in $\CC$ (note that this makes sense because $H$ is a functor).
Clearly, the slice category construction is a special case of the Grothendieck
construction for the hom-functor $H^X : \CC \to \Set$ where $H^X(A) = \hom_\CC(X, A)$.

\begin{THM}\label{crt.thm.grothendieck}
  Let $\CC$ be a locally finite category whose morphisms are mono and homsets are finite.
  Let $H : \CC \to \Set$ be a functor and assume that $\GG(\CC, H)$ is directed. Then:

  $(a)$ $t_{\GG(\CC, H)}(C, x) \le t_\CC(C)$ for all $(C, x) \in \Ob(\GG(\CC, H))$;

  $(b)$ if $\CC$ has small embedding (structural) Ramsey degrees then so does $\GG(\CC, H)$;

  $(c)$ if $\CC$ has the embedding Ramsey property then so does $\GG(\CC, H)$.
\end{THM}
\begin{proof}
  $(a)$
  Let $\DD = \GG(\CC, H)$ and let $G : \DD \to \CC$ be the forgetful functor $(C, x) \mapsto C$ and $f \mapsto f$.
  Note that $G$ is faithful. We shall use Theorem~\ref{crt.thm.comp-cocone-general} to transport
  the Ramsey property from $\CC$ to $\DD$. In order to do so we have to show that for any finite binary
  diagram $F : \Delta \to \DD$ in $\DD$ the following holds: if $GF : \Delta \to \CC$ has a compatible cocone in
  $\CC$ then $F : \Delta \to \DD$ has a compatible cocone in $\DD$.
  
  So, let $F : \Delta \to \DD$ be a finite binary diagram where $F$ takes the bottom row of $\Delta$ to $(A, a)$
  and the top row of $\Delta$ to $(B, b)$ for some $A, B \in \Ob(\CC)$, $a \in H(A)$ and $b \in H(B)$, Fig.~\ref{crt.fig.grothendieck-1},
  \begin{figure}
    \centering
    \begin{tikzcd}[column sep=small]
      & & & & & & & C
    \\
      \bullet & \ldots & \bullet & (B,b) & \ldots & (B,b) & B \arrow[ur, "{e_i}"] & \ldots & B\quad \arrow[ul, "{e_j}"']
    \\
      \bullet \arrow[urr] \arrow[u] & \ldots & \bullet \arrow[ull] \arrow[u] & (A,a) \arrow[urr, near start, "{v}"'] \arrow[u, "{u}"] & \ldots & (A,a) \arrow[ull] \arrow[u]  & A \arrow[urr, near start, "{v}"'] \arrow[u, "{u}"] & \ldots & A \arrow[ull] \arrow[u]
    \\
      & \Delta \arrow[rrr, "{F}"] & & & \DD \arrow[rrr, "{G}"] & & & \CC
    \end{tikzcd}
    \caption{A finite binary diagram in the proof of Theorem~\ref{crt.thm.grothendieck}}
    \label{crt.fig.grothendieck-1}
  \end{figure}
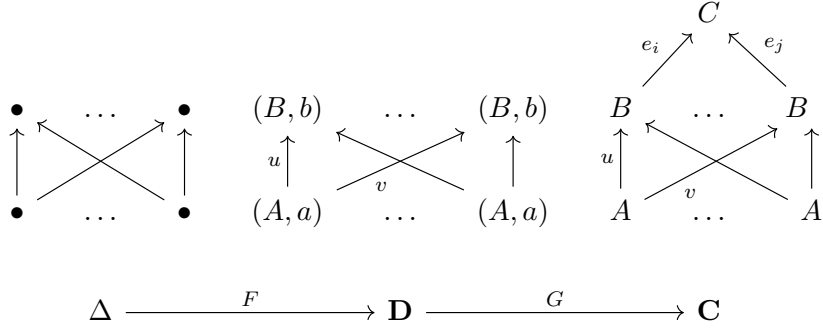
  and assume that $GF : \Delta \to \CC$ has a compatible cocone in $\CC$ with the tip
  at $C \in \Ob(\CC)$ and morphisms $e_i : B \to C$.  
  
  Let $\Delta = P \union Q$ where $P$ is the top row of the diagram and $Q$ the bottom row. Let $S \subseteq P$ be a connected component of $\Delta$
  and let us show that $H(e_i)(b) = H(e_j)(b)$ for all $i, j \in S$. Take any $i, j \in S$. Since~$S$ is a connected component of~$\Delta$
  there exist $i = t_0$, $t_1$, \ldots, $t_k = j$ in $S$,
  $s_1$, \ldots, $s_k$ in $Q$ and arrows $p_j : s_{j} \to t_{j-1}$ and $q_j : s_{j} \to t_{j}$, $1 \le j \le k$:
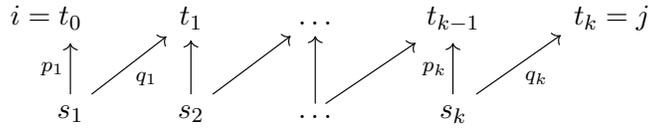
  \begin{center}
    \begin{tikzcd}
      \llap{$i = \mathstrut$}t_0 & t_1 & \ldots & t_{k-1} & t_k\rlap{$\mathstrut = j$} \\
      s_1 \arrow[u, "p_1"] \arrow[ur, "q_1"'] & s_2 \arrow[u] \arrow[ur] & \ldots \arrow[u] \arrow[ur] & s_k \arrow[u, "p_k"] \arrow[ur, "q_k"']
    \end{tikzcd}
  \end{center}
  Then
  \begin{align*}
    H(e_i)(b)
      &= H(e_{t_0})(b) = H(e_{t_0})(H(u_{p_1})(a)) = H(e_{t_0} \cdot u_{p_1})(a) = \mathstrut\\
      &= H(e_{t_1} \cdot v_{q_1})(a) = H(e_{t_1})(H(v_{q_1})(a)) = H(e_{t_1})(b),
  \end{align*}
  because $i = t_0$, $b = H(u_{p_1})(a)$ and $(e_i)_{i \in P}$ is a compatible cocone over $F$. By the same argument we now have that
  $$
    H(e_i)(b) = H(e_{t_0})(b) = H(e_{t_1})(b) = \ldots = H(e_{t_k})(b) = H(e_j)(b),
  $$
  so, letting $c = H(e_i)(b)$, we see that $(C, c)$ is the tip of a compatible cocone over this connected component of $\Delta$
  with the morphisms $e_i$.
  
  Now, assume that $\Delta$, being a finite diagram, has $n$ connected components $S_1, \ldots, S_n$.
  For each connected component $S_i$ let $c_i \in H(C)$ be constructed as above, $1 \le i \le n$.
  Since $\DD = \GG(\CC, H)$ is directed, there is a $(D, d) \in \Ob(\DD)$ and arrows $h_i : (C, c_i) \to (D, d)$, $1 \le i \le n$.

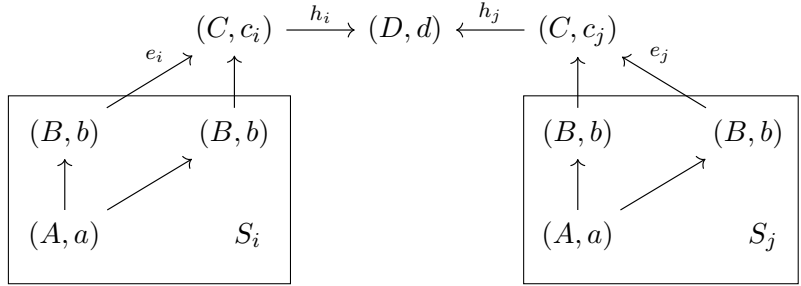
  \begin{center}
    \begin{tikzcd}[execute at end picture={
            \draw (1.1,-2) rectangle (4.75,0.5);
            \draw (-5.75,-2) rectangle (-2,0.5);
        }]
    & (C, c_i) \arrow[r, "h_i"] & (D, d) & (C, c_j) \arrow[l, "h_j"'] & & 
  \\
    (B, b) \arrow[ur, "e_i" near end] & (B, b) \arrow[u] &  & (B, b) \arrow[u] & (B, b) \arrow[ul, "e_j"' near end]
  \\
    (A, a) \arrow[u] \arrow[ur]  & \rlap{$S_i$} & & (A, a) \arrow[ur] \arrow[u] & \rlap{$S_j$}
    \end{tikzcd}
  \end{center}
  
  Then $(D, d)$ is the tip of a compatible cocone over $F$ in $\DD$, where the morphisms are of the form
  $h_i \cdot e : (B, b) \to (D, d)$.

  \medskip
  
  $(b)$ follows from $(a)$ and Proposition~\ref{rdbas.prop.sml}.
  
  $(c)$ follows from $(b)$.
\end{proof}

The proof of the following theorem is similar to the proof of Theorem~\ref{crt.thm.grothendieck}.
We, therefore, provide only an outline.

\begin{COR}\label{crt.cor.slice-cat}
  Let $\CC$ be a category with amalgamation whose homsets are finite and morphisms are mono.
  Let $X \in \Ob(\CC)$ be arbitrary and $\DD = \slice X\CC$. Then

  $(a)$ $t_{\DD}(x_C, C) \le t_\CC(C)$ for all $(x_C, C) \in \Ob(\DD)$;

  $(b)$ if $\CC$ has small embedding (structural) Ramsey degrees then so does~$\DD$;

  $(c)$ if $\CC$ has the embedding Ramsey property then so does $\DD$.
\end{COR}
\begin{proof}
  Let $H : \CC \to \Set$ be the hom-functor $H(A) = \hom_\CC(X, A)$ and $H(f) = f \cdot -$.
  Let $G : \DD \to \CC$ be the forgetful functor $(x_C, C) \mapsto C$ and $f \mapsto f$.
  Note that $G$ is faithful so that we can use Theorem~\ref{crt.thm.comp-cocone-general} to transport
  the Ramsey property from $\CC$ to $\DD$.
  
  Let $F : \Delta \to \DD$ be a finite binary diagram where $F$ takes the bottom row of $\Delta$ to $(x_A, A)$
  and the top row of $\Delta$ to $(x_B, B)$ for some $A, B \in \Ob(\CC)$, $x_A \in \hom_\CC(X, A)$ and
  $x_B \in \hom_\CC(X, B)$ and assume that $GF : \Delta \to \CC$ has a compatible cocone in $\CC$ with the tip
  at $C \in \Ob(\CC)$ and morphisms $e_i : B \to C$.  

  Let $\Delta = P \union Q$ where $P$ is the top row of the diagram and $Q$ the bottom row.
  Let $S \subseteq P$ be a connected component of $\Delta$. Then as in the proof of Theorem~\ref{crt.thm.grothendieck}
  we see that $e_i \cdot x_B = e_j \cdot x_B$ for all $i, j \in S$.
  Therefore, letting $x_C = e_{i_0} \cdot x_B$ for an arbitrary but fixed $i_0 \in \SS$
  we see that $(x_C, C)$ is the tip of a compatible cocone over this connected component
  of $\Delta$ with the morphisms~$e_i$, $i \in S$.
  
  Now, assume that $\Delta$, being a finite diagram, has $n$ connected components $S_1, \ldots, S_n$.
  For each connected component $S_i$ let $x^i_C \in \hom_\CC(X, C)$ be constructed as above, $1 \le i \le n$.
  Since $\CC$ has amalgamation there exist a $D \in \Ob(\CC)$ and morphisms $h_i \in \hom_\CC(C, D)$, $1 \le i \le n$, such that
  
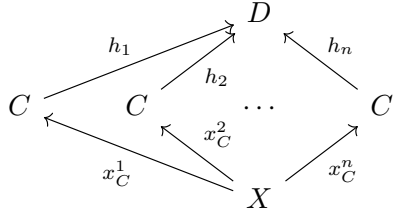
\begin{center}
    \begin{tikzcd}
        &   &  D     &   \\
      C \arrow[urr, "h_1"] & C \arrow[ur, "h_2"'] & \cdots & C \arrow[ul, "h_n"']\\
        &   &  X \arrow[ull, "x_C^1"] \arrow[ul, "x_C^2"']  \arrow[ur, "x_C^n"'] & 
    \end{tikzcd}
  \end{center}
  Let $x_D = h_1 \cdot x_C^1 = \ldots = h_n \cdot x_C^n$. Then, clearly, $h_i : (x_C^i, C) \to (x_D, D)$, $1 \le i \le n$.
  Therefore, $(x_D, D)$ is the tip of a compatible cocone over $F$ in~$\DD$:
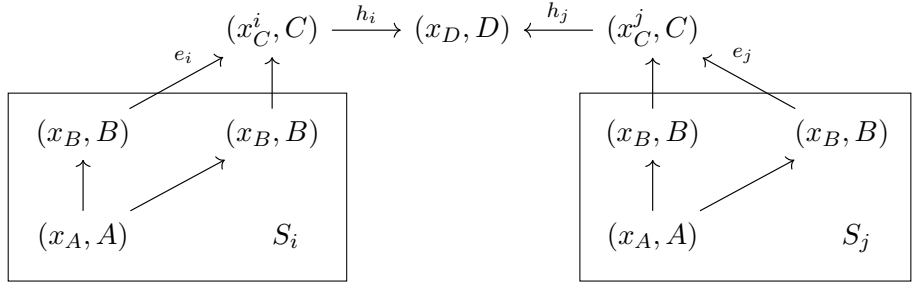
  \begin{center}
    \begin{tikzcd}[execute at end picture={
            \draw (1.1,-2) rectangle (5.5,0.5);
            \draw (-6.5,-2) rectangle (-2,0.5);
        }]
    & (x_C^i, C) \arrow[r, "h_i"] & (x_D, D) & (x_C^j, C) \arrow[l, "h_j"'] & & 
  \\
    (x_B, B) \arrow[ur, "e_i" near end] & (x_B, B) \arrow[u] &  & (x_B, B) \arrow[u] & (x_B, B) \arrow[ul, "e_j"' near end]
  \\
    (x_A, A) \arrow[u] \arrow[ur]  & \rlap{$S_i$} & & (x_A, A) \arrow[ur] \arrow[u] & \rlap{$S_j$}
    \end{tikzcd}
  \end{center}

  $(b)$ follows from $(a)$ and Proposition~\ref{rdbas.prop.sml}.
  
  $(c)$ follows from $(b)$.
\end{proof}

\begin{COR}
  Let $\Theta$ be a first-order language and let $\KK$ be an amalgamation class of finite $\Theta$-structures.
  Let $c_1, \ldots, c_n \notin \Theta$ be new constant symbols, let
  $\Theta' = \Theta \union \{c_1, \ldots, c_n\}$ and let $\KK'$ be the class of $\Theta'$-structures
  of the form $(\calA, a_1, \ldots, a_n)$ where $\calA \in \KK$ and $a_1, \ldots, a_n \in A$,
  the underlying set of $\calA$. Then

  $(a)$ $t_{\KK'}(\calA, a_1, \ldots, a_n) \le t_\KK(\calA)$ for all $\calA \in \KK$ and $a_1, \ldots, a_n \in A$;

  $(b)$ if $\KK$ has small embedding (structural) Ramsey degrees then so does~$\KK'$;

  $(c)$ if $\KK$ has the embedding Ramsey property then so does $\KK'$.
\end{COR}
\begin{proof}
  As usual, we consider both $\KK$ and $\KK'$ as categories with embeddings as morphisms.
  Note that $\KK'$ partitions into a disjoint union of slice categories according to the
  isomorphism type of the structure $\calA[a_1, \ldots, a_n]$ where $(\calA, a_1, \ldots, a_n) \in \KK'$.

  $(a)$ Take any $(\calA, a_1, \ldots, a_n)$ and $(\calB, b_1, \ldots, b_n)$ in $\KK'$ such that
  $(\calA, a_1, \ldots, a_n) \hookrightarrow (\calB, b_1, \ldots, b_n)$. Let $k \in \NN$ be arbitrary
  and let $t = t_\KK(\calA)$. Let us show that there is a $(\calC, c_1, \ldots, c_n) \in \KK'$ such that
  \begin{equation}\label{crt.eq.adding-consts-0}
    (\calC, c_1, \ldots, c_n) \longrightarrow (\calB, b_1, \ldots, b_n)^{(\calA, a_1, \ldots, a_n)}_{k, t} \text{ in } \KK'.
  \end{equation}
  Let $\calX = \calA[a_1, \ldots, a_n]$. Then, clearly, $\calX \cong \calB[b_1, \ldots, b_n]$.
  Let $f_\calA : \calX \to \calA$ and $f_\calB : \calX \to \calB$ be the unique embeddings satisfying
  $f_\calA(a_i) = a_i$ and $f_\calB(a_i) = b_i$, $1 \le i \le n$.
  Then $(f_\calA, \calA)$ and $(f_\calB, \calB)$ are objects of $\slice\calX\KK$.
  Note that
  \begin{equation}\label{crt.eq.adding-consts}
    \hom_{\KK'}\big(
      (\calA, a_1, \ldots, a_n), (\calB, b_1, \ldots, b_n)
    \big)
    =
    \hom_{\slice\calX\KK}\big(
      (f_\calA, \calA), (f_\calB, \calB)
    \big),
  \end{equation}
  so $(f_\calA, \calA) \to (f_\calB, \calB)$. By Corollary~\ref{crt.cor.slice-cat}~$(a)$ we have that
  there is an $(f_\calC, \calC)$ in $\slice\calX\KK$
  such that $(f_\calC, \calC) \longrightarrow (f_\calB, \calB)^{(f_\calA, \calA)}_{k,t}$ in $\slice\calX\KK$.
  Let $f_\calC(a_i) = c_i$, $1 \le i \le n$. Then, as an immediate consequence of \eqref{crt.eq.adding-consts}
  we get~\eqref{crt.eq.adding-consts-0}.

  $(b)$ follows from $(a)$ and Proposition~\ref{rdbas.prop.sml}.
  
  $(c)$ follows from $(b)$.
\end{proof}


\section{Declarations}

\paragraph{Funding.}
This work was supported by the Science Fund of the Republic of Serbia, Grant No.~7750027:
\textit{Set-theoretic, model-theoretic and Ramsey-theoretic phenomena in mathematical structures: similarity and diversity -- SMART}.

\paragraph{Data Availability.}
Data sharing not applicable to this article as no data\-sets were generated or analysed during
the current study.

\paragraph{Conflict of Interest.}
The author declares no conflict of interest.

\end{document}